\documentclass[10pt]{amsart}

\usepackage[all]{xy}

\usepackage{latexsym}

\usepackage{amssymb}
\usepackage{amsfonts}
\usepackage{amscd}
\usepackage{amsmath,amsthm}

\newtheorem{lemma}{Lemma}[section]

\newtheorem{theorem}[lemma]{Theorem}

\newtheorem*{introthm}{Theorem}

\newtheorem{defn}[lemma]{Definition}

\newtheorem{prop}[lemma]{Proposition}
\newtheorem{thm}[lemma]{Theorem}
\newtheorem{cor}[lemma]{Corollary}

{

}

\theoremstyle{definition}

\theoremstyle{remark}

\newcommand{\Ker}{\operatorname{ker }\Gamma_{F}}
\newcommand{\Cok}{\operatorname{cok }\Gamma_{F}}

\numberwithin{equation}{section}

\begin{document}

\title{A structure theorem for $\mathbb{P}^{1}-\operatorname{Spec }k$-bimodules}
\author{A. Nyman}
\address{Department of Mathematics, 516 High St, Western Washington University, Bellingham, WA 98225-9063}
\email{adam.nyman@wwu.edu}
\keywords{}
\date{\today}
\thanks{2010 {\it Mathematics Subject Classification. } Primary 18A25; Secondary 14A22}

\maketitle

\begin{abstract}
Let $k$ be an algebraically closed field.  Using the Eilenberg-Watts theorem over schemes \cite{N}, we determine the structure of $k$-linear right exact direct limit and coherence preserving functors from the category of quasi-coherent sheaves on $\mathbb{P}_{k}^{1}$ to the category of vector spaces over $k$.  As a consequence, we characterize those functors which are integral transforms.
\end{abstract}

\pagenumbering{arabic}

\section{Introduction}
Throughout, we work over an algebraically closed field $k$.  For any $k$-algebra $R$, we let ${\sf Mod }R$ denote the category of right $R$-modules, and for any $k$-scheme $X$, we let ${\sf Qcoh }X$ denote the category of quasi-coherent sheaves on $X$.  The purpose of this note is to prove the following (Theorem \ref{thm.extension}):
\begin{introthm}
If $F:{\sf Qcoh }\mathbb{P}^{1}_{k} \rightarrow {\sf Mod }k$ is a $k$-linear right exact direct limit and coherence preserving functor then there exists a coherent torsion sheaf $\mathcal{T}$ and nonnegative integers $n$, $n_{i}$ such that
$$
F \cong \oplus_{i=-n}^\infty {H}^{1}(\mathbb{P}^{1},(-)(i))^{\oplus n_{i}} \oplus H^{0}(\mathbb{P}^{1},-\otimes \mathcal{T}).
$$
\end{introthm}
Our motivation for determining the structure of such functors is twofold.

Our first motivation comes from the characterization of direct limit preserving right exact functors between module categories of $k$-algebras carried out independently by Eilenberg \cite{E} and Watts \cite{W}.  The characterization (known as the Eilenberg-Watts theorem) says that if $R$ and $S$ are $k$-algebras and $F:{\sf Mod }R \rightarrow {\sf Mod S}$ is a $k$-linear right exact direct limit preserving functor, then $F$ is an integral transform, i.e. is isomorphic to $-\otimes_{R}M$, where $M$ is a $k$-central $R-S$-bimodule.  In \cite{N}, $k$-linear direct limit preserving right exact functors between categories of quasi-coherent sheaves on schemes are studied.  The main result in \cite{N} is that such functors are {\it almost} integral transforms.  To describe the result more precisely, we introduce some notation.  We let $X$ be a quasi-compact and separated scheme and we let $Y$ be a separated scheme.  We let ${\sf Bimod}_{k}(X-Y)$ denote the category of $k$-linear right exact direct limit preserving functors from ${\sf Qcoh }X$ to ${\sf Qcoh }Y$ and we let
$$
W:{\sf Bimod}_{k}(X-Y) \rightarrow {\sf Qcoh }X  \times  Y,
$$
denote the Eilenberg-Watts functor (see \cite[Section 5]{N} for details).  We let $\operatorname{pr}_{1,2}:X \times Y \rightarrow X,Y$ denote standard projection maps and we recall a

\begin{defn} \label{defn.totallyglobal}
A $k$-linear functor $F: {\sf Qcoh }X \rightarrow {\sf Qcoh }Y$ is {\it totally global} if for every
open immersion $u:U \rightarrow X$ with $U$ affine, $Fu_{*}=0$.
\end{defn}
The Eilenberg-Watts theorem over schemes is the following

\begin{theorem} \label{thm.watts} \cite[Section 6]{N}
If $F$ is an object in ${\sf Bimod}_{k}(X-Y)$ then there exists a natural transformation
$$
\Gamma_{F}:F \longrightarrow \operatorname{pr}_{2*}(\operatorname{pr}_{1}^{*}-\otimes_{\mathcal{O}_{X \times Y}}W(F))
$$
such that $\operatorname{ker }\Gamma_{F}$
and $\operatorname{cok }\Gamma_{F}$ are totally global.  Furthermore, if $X$ is affine, then $\Gamma_{F}$ is an isomorphism.
\end{theorem}
The result leaves open the question of whether there is a more precise description of the structure of such functors when $X$ is not affine.  Our structure theorem addresses this question in case $X=\mathbb{P}^{1}_{k}$ and $Y=\operatorname{Spec }k$.  It follows immediately from the structure theorem that a coherence preserving functor $F$ in ${\sf Bimod}_{k}(\mathbb{P}^{1}_{k}-\operatorname{Spec }k)$
is an integral transform, i.e. $\Gamma_{F}$ is an isomorphism, if and only if $F$ is exact on short exact sequences of vector bundles (Corollary \ref{cor.vectorbundles}).

Our second motivation for proving the structure theorem is related to non-commutative algebraic geometry.  The objects of study in this subject are non-commutative spaces, which are certain abelian categories which serve as generalizations of categories of quasi-coherent sheaves on quasi-compact and separated schemes.  Maps between such spaces are taken to be equivalence classes of adjoint pairs of functors between the spaces (see \cite[Definition 2.3]{Smith} for a more precise discussion).  The fact that this notion of map is so general has the benefit of increased flexibility but also has the drawback of deviating too far from its commutative origin.  In order to obtain a more restricted notion of map in the non-commutative setting, it may thus be instructive to study the question of when a map between categories of quasi-coherent sheaves on schemes is isomorphic to $(f^{*},f_{*})$ for some morphism $f$ of schemes.  It follows from our structure theorem that a coherence preserving functor $F$ in ${\sf Bimod}_{k}(\mathbb{P}^{1}_{k}-\operatorname{Spec }k)$ is isomorphic to $f^{*}$ where $f:\operatorname{Spec }k \rightarrow \mathbb{P}_{k}^{1}$ is a morphism of schemes if and only if $F$ is exact on short exact sequences of vector bundles and $\operatorname{dim }_{k}(F(\mathcal{O}(n)))=1$ for some $n$ (Corollary \ref{cor.morphism}).

\medskip

{\it Notation and Conventions}: In addition to the notation introduced above, we let ${\sf Funct}_{k}({\sf Qcoh }X,{\sf Qcoh }Y)$ denote the abelian category of $k$-linear functors from ${\sf Qcoh }X$ to ${\sf Qcoh }Y$.  We routinely use the fact that in this category, a kernel of a natural transformation $\Upsilon:F \rightarrow G$ is the functor which assigns to an object $\mathcal{M}$ the kernel of $\Upsilon_{\mathcal{M}}$ and which assigns to a morphism $\phi:\mathcal{M} \rightarrow \mathcal{N}$ the induced morphism $\operatorname{ker }\Upsilon_{\mathcal{M}} \rightarrow \operatorname{ker }\Upsilon_{\mathcal{N}}$.  A cokernel of $\Upsilon$ can be defined similarly.  We denote the full subcategory of ${\sf Bimod}_{k}(X-Y)$ consisting of coherence preserving functors by ${\sf bimod}_{k}(X-Y)$.

We will routinely invoke $\Gamma_{F}$, $W(F)$, $\Ker$ and $\Cok$ from Theorem \ref{thm.watts} without explicit reference to Theorem \ref{thm.watts}.  We write $\mathbb{P}^{1}$ for $\mathbb{P}_{k}^{1}$.  All unadorned tensor products are over $\mathcal{O}_{\mathbb{P}^{1}}$, and we write $\mathcal{O}(i)$ for $\mathcal{O}_{\mathbb{P}^{1}}(i)$.

We shall abuse notation by identifying ${\sf Qcoh }(\operatorname{Spec }k)$ with the category ${\sf Mod }k$.

Other notation and conventions will be introduced locally.
\medskip

\section{The structure of $W(F)$} \label{section.wf}
 Although some of the properties of the Eilenberg-Watts functor $W$ will play a crucial role in the proof of our main result, the details of the construction of $W$, which are somewhat complicated, will not be needed in what follows.  For this reason, we refer the reader to \cite[Section 5]{N} for the definition of $W$.  Our main goal in this section is to prove that if $F \in {\sf bimod }_{k}(\mathbb{P}^{1}-\operatorname{Spec }k)$ then $W(F)$ is coherent torsion.

\begin{prop} \label{prop.coherent}
If $F \in {\sf bimod }_{k}(\mathbb{P}^{1}-\operatorname{Spec }k)$ then $W(F)$ is coherent torsion.
\end{prop}

\begin{proof}
We claim that, to prove the proposition, it suffices to prove that if $p \in \mathbb{P}^{1}$ is a closed point, $U$ is the complement of $p$ in $\mathbb{P}^{1}$ and $u:U \rightarrow \mathbb{P}^{1}$ is inclusion, then $F(u_{*}u^{*}\mathcal{O}_{\mathbb{P}^{1}})$ is finite dimensional.  To prove the claim, we first note that
$u^{*}W(F) \cong W(Fu_{*})$ by \cite[Proposition 5.2]{N}, where we have abused notation by identifying $\mathbb{P}^{1} \times_{\operatorname{Spec }k} \operatorname{Spec }k$ with $\mathbb{P}^{1}$.  By the proof of \cite[Proposition 2.2]{N}, $W(Fu_{*})$ has $k$-module structure $F(u_{*}u^{*}\mathcal{O}_{\mathbb{P}^{1}})$.  Therefore, the finite dimensionality of $F(u_{*}u^{*}\mathcal{O}_{\mathbb{P}^{1}})$ would imply the finite dimensionality of $u^{*}W(F)$.  Since $p$ is arbitrary, the fact that $W(F)$ is coherent, and hence torsion, would follow.

We now proceed to prove that $F(u_{*}u^{*}\mathcal{O}_{\mathbb{P}^{1}})$ is finite dimensional, which will establish the result.  To this end, by \cite[Corollary 1.9]{Hart2}, there is a short exact sequence
$$
0 \rightarrow \mathcal{O}_{\mathbb{P}^{1}} \rightarrow u_{*}u^{*}\mathcal{O}_{\mathbb{P}^{1}} \rightarrow \mathcal{H}_{p}^{1}(\mathcal{O}_{\mathbb{P}^{1}}) \rightarrow 0
$$
where $\mathcal{H}_{p}^{1}:{\sf Qcoh }\mathbb{P}^{1} \rightarrow {\sf Qcoh }\mathbb{P}^{1}$ is the functor which sends a quasi-coherent sheaf $\mathcal{F}$ to the first cohomology of $\mathbb{P}^{1}$ with coefficients in $\mathcal{F}$ and support in $p$.  Applying the right exact functor $F$ to this sequence yields an exact sequence
$$
F(\mathcal{O}_{\mathbb{P}^{1}}) \rightarrow F(u_{*}u^{*}\mathcal{O}_{\mathbb{P}^{1}}) \rightarrow F(\mathcal{H}_{p}^{1}(\mathcal{O}_{\mathbb{P}^{1}})) \rightarrow 0.
$$
Thus, to prove that $F(u_{*}u^{*}\mathcal{O}_{\mathbb{P}^{1}})$ is finite dimensional, it suffices to prove that $F(\mathcal{H}_{p}^{1}(\mathcal{O}_{\mathbb{P}^{1}}))$ is zero.

To prove that $F(\mathcal{H}_{p}^{1}(\mathcal{O}_{\mathbb{P}^{1}}))$ vanishes, we first recall that
if $\mathcal{I}$ is the ideal sheaf of $p$ and we let $\mathcal{O}_{n} := \mathcal{O}_{\mathbb{P}^{1}}/\mathcal{I}^{n}$ for $n \geq 1$, then
$$
\mathcal{H}_{p}^{1}(\mathcal{O}_{\mathbb{P}^{1}}) \cong \underset{\rightarrow}{\mbox{lim }}
\mathcal{E}xt_{\mathcal{O}_{\mathbb{P}^{1}}}^{1}(\mathcal{O}_{n},\mathcal{O}_{\mathbb{P}^{1}})
$$
where the direct system is induced by the canonical quotient maps $\mathcal{O}_{n} \rightarrow \mathcal{O}_{m}$ for $n \geq m$ (\cite[Theorem 2.8]{Hart2}).  Thus, the proposition will follow if we can show that $F(\underset{\rightarrow}{\mbox{lim }}
\mathcal{E}xt_{\mathcal{O}_{\mathbb{P}^{1}}}^{1}(\mathcal{O}_{n},\mathcal{O}_{\mathbb{P}^{1}})) \cong \underset{\rightarrow}{\mbox{lim }}F(
\mathcal{E}xt_{\mathcal{O}_{\mathbb{P}^{1}}}^{1}(\mathcal{O}_{n},\mathcal{O}_{\mathbb{P}^{1}}))$
is zero.  To prove this final fact, we first leave it as a straightforward exercise to show that $\underset{\rightarrow}{\mbox{lim }}
\mathcal{E}xt_{\mathcal{O}_{\mathbb{P}^{1}}}^{1}(\mathcal{O}_{n},\mathcal{O}_{\mathbb{P}^{1}})$ is isomorphic to the direct limit of the direct system $(\mathcal{O}_{i},\mu_{ij})$ where for $i < j$, $\mu_{ij}:\mathcal{O}_{i} \rightarrow \mathcal{O}_{j}$ is the inclusion of $\mathcal{O}_{i}$ as the kernel of the canonical quotient $\mathcal{O}_{j} \rightarrow \mathcal{O}_{j-i}$ and $\mu_{ii}$ is the identity map.  We next note that since there are epimorphisms $\mathcal{O}_{j} \rightarrow \mathcal{O}_{i}$ for all $i \leq j$ and $\mathcal{O}_{\mathbb{P}^{1}} \rightarrow \mathcal{O}_{i}$ for all $i$ and since $F$ is right exact, $\mbox{dim}_{k}F(\mathcal{O}_{i}) \leq \mbox{dim}_{k}F(\mathcal{O}_{j}) \leq \mbox{dim}_{k}F(\mathcal{O}_{\mathbb{P}^{1}})$ for all $i \leq j$.  It follows from this that there exists a natural number $n_{0}$ such that $\mbox{dim}_{k}F(\mathcal{O}_{i})=\mbox{dim}_{k}F(\mathcal{O}_{j})$ for all $i,j \geq n_{0}$.  Therefore, for $i,j$ such that $n_{0} \leq i \leq j$, the canonical quotient map $\mathcal{O}_{j} \rightarrow \mathcal{O}_{i}$ induces an isomorphism $F(\mathcal{O}_{j}) \rightarrow F(\mathcal{O}_{i})$, and thus, in this case, the canonical inclusion $\mu_{j-i,j}$ induces the zero map $F(\mathcal{O}_{j-i}) \rightarrow F(\mathcal{O}_{j})$.  It follows that $\underset{\rightarrow}{\mbox{lim }}F(\mathcal{O}_{n})$ is zero, so that $F(\mathcal{H}_{p}^{1}(\mathcal{O}_{\mathbb{P}^{1}}))$ is zero.  The result follows.
\end{proof}

\begin{cor} \label{cor.cokvanish}
If $F \in {\sf bimod }_{k}(\mathbb{P}^{1}-\operatorname{Spec }k)$, then $\Cok=0$.  Thus, there is a short exact sequence
$$
0 \rightarrow \Ker \rightarrow F \overset{\Gamma_{F}}{\rightarrow} H^{0}(\mathbb{P}^{1},-\otimes W(F)) \rightarrow 0
$$
in ${\sf Funct}_{k}({\sf Qcoh }\mathbb{P}^{1},{\sf Qcoh }(\operatorname{Spec }k))$.
\end{cor}

\begin{proof}
Since $W(F)$ is coherent torsion by Proposition \ref{prop.coherent}, the functor $H^{0}(\mathbb{P}^{1},-\otimes W(F))$ is right exact and commutes with direct limits.  Thus, a diagram chase establishes that $\Cok$ is right exact and commutes with direct limits as well.  Since $k$ is algebraically closed, it follows from \cite[Corollary 7.3]{N} and \cite[Theorem 7.12]{N} that there exist integers $m,
n_{i} \geq 0$ such that
$$
\Cok \cong \oplus_{i=-m}^\infty {H}^{1}(\mathbb{P}^{1},(-)(i))^{\oplus n_{i}}.
$$
Thus, if $\Cok$ were nonzero, $\Cok(\mathcal{O}(i))$ would have arbitrarily large dimension for negative $i \in \mathbb{Z}$.  On the other hand, $H^{0}(\mathbb{P}^{1},\mathcal{O}(i) \otimes W(F))$ has constant dimension.  Since there is an epimorphism $H^{0}(\mathbb{P}^{1},\mathcal{O}(i) \otimes W(F)) \rightarrow \Cok(\mathcal{O}(i))$, we conclude that $\Cok=0$.
\end{proof}

We end this section by describing a situation we shall encounter several times throughout this note.  Suppose $F \in {\sf Funct}_{k}({\sf Qcoh }X, {\sf Qcoh }(\operatorname{Spec }k))$ is right exact.  If
\begin{equation} \label{eqn.ses}
0 \rightarrow \mathcal{M} \rightarrow \mathcal{N} \rightarrow \mathcal{P} \rightarrow 0
\end{equation}
is a short exact sequence in ${\sf Qcoh }X$, $\mathcal{G}$ is an object in ${\sf Qcoh }X$ and $\Upsilon:F \rightarrow H^{0}(X,- \otimes_{\mathcal{O}_{X}} \mathcal{G})$ is a natural transformation, then we have a commutative (but not necessarily exact) diagram
\begin{equation} \label{eqn.diag1}
\begin{CD}
& & F(\mathcal{M}) & \rightarrow & F(\mathcal{N}) & \rightarrow & F(\mathcal{P}) & \rightarrow & 0 \\
& & @VVV @VVV @VVV & \\
0 & \rightarrow & H^{0}(X,\mathcal{M} \otimes_{\mathcal{O}_{X}} \mathcal{G}) & \rightarrow & H^{0}(X,\mathcal{N} \otimes_{\mathcal{O}_{X}} \mathcal{G}) & \rightarrow & H^{0}(X,\mathcal{P} \otimes_{\mathcal{O}_{X}} \mathcal{G}) &
\end{CD}
\end{equation}
induced by $\Upsilon$.  The top row is always exact, while the bottom row is exact in case $\mathcal{P}$ is flat.  If the bottom row is exact, the Snake Lemma implies that there is an exact sequence
\begin{equation} \label{eqn.snake}
\operatorname{ker }\Upsilon(\mathcal{M}) \rightarrow \operatorname{ker }\Upsilon (\mathcal{N}) \rightarrow \operatorname{ker }\Upsilon(\mathcal{P}) \rightarrow
\end{equation}
$$\operatorname{cok }\Upsilon(\mathcal{M}) \rightarrow \operatorname{cok }\Upsilon(\mathcal{N}) \rightarrow \operatorname{cok }\Upsilon(\mathcal{P})
$$
in ${\sf Qcoh }(\operatorname{Spec }k)$.

\section{The structure of $\operatorname{ker }\Gamma_{F}$} \label{section.ker}
Throughout this section, suppose $F \in {\sf bimod }_{k}(\mathbb{P}^{1}-\operatorname{Spec }k)$ is such that $W(F)$ is nonzero.  Since, by Proposition \ref{prop.coherent}, $W(F)$ is coherent torsion, it has finite support.  For the remainder of this note, we let $p \in \mathbb{P}^{1}$ be a closed point outside the support of $W(F)$ and we let $q \in \mathbb{P}^{1}$ be a closed point in the support of $W(F)$.

Let $A=k[x_{0},x_{1}]$ denote the polynomial ring with the usual grading, let $[-]$ denote the shift functor in the category of graded $A$ modules, and let $f,g:A[-1] \rightarrow A$ denote multiplication by linear forms whose corresponding morphisms in ${\sf Qcoh}\mathbb{P}^{1}$, $\alpha_{0}:\mathcal{O}(-1) \rightarrow \mathcal{O}$ and $\beta_{0}:\mathcal{O}(-1) \rightarrow \mathcal{O}$, have cokernels $k(p)$ and $k(q)$, respectively. For any morphism $\delta_{0}:\mathcal{O}(-1) \rightarrow \mathcal{O}$ corresponding to multiplication by a linear form $h:A[-1] \rightarrow A$, we define $\delta_{i}:\mathcal{O}(i-1) \rightarrow \mathcal{O}(i)$ to be the morphism corresponding to the $i$th shift of $h$.

\begin{lemma} \label{lemma.alphabetaepi}
\begin{enumerate}
\item The morphism $\Ker(\alpha_{i})$ is epic.  Thus, there is a nonnegative integer $m$ such that $\operatorname{dim }_{k}(\Ker(\mathcal{O}(i)))=m$ for all sufficiently large $i$.

\item The image of the morphism $\Ker(\beta_{j-1})$ contains
$$
(\Ker(\alpha_{j}))^{-1}(\operatorname{im }(\Ker(\beta_{j}))).
$$
Therefore, if $\Ker(\beta_{i})$ is epic then $\Ker(\beta_{j})$ is epic for all $j \leq i$.
\end{enumerate}
\end{lemma}

\begin{proof}
The sequence (\ref{eqn.snake}) in the case (\ref{eqn.ses}) is
$$
0 \rightarrow \mathcal{O}(i-1) \overset{\alpha_{i}}{\rightarrow} \mathcal{O}(i) \rightarrow k(p) \rightarrow 0,
$$
$\Upsilon=\Gamma_{F}$, and $\mathcal{G}=W(F)$, has an exact subsequence
$$
\Ker(\mathcal{O}(i-1)) \overset{\Ker(\alpha_{i})}{\rightarrow} \Ker(\mathcal{O}(i)) \rightarrow 0
$$
since $\Ker$ is totally global by Theorem \ref{thm.watts}.  Therefore, $\Ker(\alpha_{i})$ is epic.  The second statement in (1) follows from this.

To prove (2), we note that in the case (\ref{eqn.ses}) is
\begin{equation} \label{eqn.vect}
0 \longrightarrow \mathcal{O}(j-2)
\overset{(\alpha_{j-1},-\beta_{j-1})}{\longrightarrow}
\mathcal{O}(j-1)^{\oplus 2}
\overset{\beta_{j}+\alpha_{j}}{\longrightarrow} \mathcal{O}(j)
\longrightarrow 0,
\end{equation}
$\Upsilon=\Gamma_{F}$, and $\mathcal{G}=W(F)$, the sequence (\ref{eqn.snake}) has an exact subsequence
\begin{equation} \label{eqn.kerapplied}
\Ker(\mathcal{O}(j-2))
\overset{(\Ker(\alpha_{j-1}),-\Ker(\beta_{j-1}))}{\longrightarrow}
\end{equation}
$$
\Ker(\mathcal{O}(j-1))^{\oplus 2}
\overset{\Ker(\beta_{j})+\Ker(\alpha_{j})}{\longrightarrow} \Ker(\mathcal{O}(j))
\longrightarrow 0
$$
by Corollary \ref{cor.cokvanish}.  If $a \in (\Ker(\alpha_{j}))^{-1}(\operatorname{im }(\Ker(\beta_{j})))$, then there exists a $b \in \Ker(\mathcal{O}(j-1))$ such that $\Ker(\alpha_{j})(a)=\Ker(\beta_{j})(b)$. Therefore, by exactness of (\ref{eqn.kerapplied}), $(b,-a)$ is in the image of $(\Ker(\alpha_{j-1}),-\Ker(\beta_{j-1}))$, so that $a$ is in the image of $\Ker(\beta_{j-1})$.  The first part of (2) follows.  To prove the second part of (2), we note that if $\Ker(\beta_{i})$ is epic, then the first part of (2) implies that $(\Ker(\alpha_{i}))^{-1}(\operatorname{im }(\Ker(\beta_{i})))=\Ker(\mathcal{O}(i-1))$ so that $\Ker(\beta_{i-1})$ is epic as well.
\end{proof}

\begin{cor} \label{cor.kernelstructure}
If $m=0$ in Lemma \ref{lemma.alphabetaepi}(1), then $\Ker(\beta_{i})$ is epic for all $i$.  It follows that there exist integers $n,
n_{i} \geq 0$ such that
$$
\Ker \cong \oplus_{i=-n}^\infty {H}^{1}(\mathbb{P}^{1},(-)(i))^{\oplus n_{i}}.
$$
\end{cor}

\begin{proof}
Either $\Ker(\mathcal{O}(n)) = 0$ for all $n$ (in which case the last result is clear) or there exists $n_{0}$ minimal such that $\Ker(\mathcal{O}(i))=0$ for all $i \geq n_{0}$.  Therefore, $\Ker(\beta_{i})$ is trivially epic for all $i \geq n_{0}$.  The first result now follows immediately from Lemma \ref{lemma.alphabetaepi}(2).

To prove the second result, we prove that $\Ker$ is admissible \cite[Definition 7.1]{N}.  To this end, we note that it is half-exact on vector bundles by the exactness of (\ref{eqn.snake}) in case (\ref{eqn.ses}) is an exact sequence of vector bundles, $\Upsilon=\Gamma_{F}$ and $\mathcal{G}=W(F)$.  We next note that $\Ker(\alpha_{i})$ is epic for all $i$ by Lemma \ref{lemma.alphabetaepi}(1) while $\Ker(\beta_{i})$ is epic for all $i$ by the argument in the first paragraph. Since $p$ is an arbitrary point outside the support of $W(F)$ while $q$ is an arbitrary point in the support of $W(F)$, and since any nonzero map $\delta: \mathcal{O}(i-1) \rightarrow \mathcal{O}(i)$ has cokernel $k(r)$ for some closed point $r \in \mathbb{P}^{1}$, it follows that $\Ker(\delta)$ is epic.  Therefore, since $k$ is algebraically closed, $\Ker(\gamma)$ is epic for all nonzero $\gamma \in \operatorname{Hom}_{\mathcal{O}_{\mathbb{P}^{1}}}(\mathcal{O}(m),\mathcal{O}(n))$.

It remains to check that $\Ker$ commutes with direct limits.  This follows from the fact that both $F$ and $H^{0}(\mathbb{P}^{1},-\otimes W(F))$ commute with direct limits.  Therefore, \cite[Theorem 7.12]{N} implies the second result.
\end{proof}

For the remainder of this section, we will prove that $m$ from Lemma \ref{lemma.alphabetaepi}(1) must equal $0$. Our strategy will be to use the fact that if $m>0$, we can embed a functor which is not totally global into $\Ker$, thus obtaining a contradiction.  Our strategy is made possible by the fact that natural transformations $\Omega:F \rightarrow G$ between direct limit preserving functors in ${\sf Funct}_{k}({\sf Qcoh }\mathbb{P}^{1}, {\sf Qcoh }(\operatorname{Spec }k))$ such that $G$ is totally global may be constructed inductively.  The precise result (used implicitly in the proof of \cite[Proposition 7.6]{N}) is given by the following

\begin{lemma} \label{lemma.inductive}
Let $F$ and $G$ be direct limit preserving $k$-linear functors from ${\sf Qcoh }\mathbb{P}^{1}$ to ${\sf Qcoh }Y$ such that $G$ is totally global.  Suppose for all $n \in \mathbb{Z}$, morphisms $\underline{\Omega}_{\mathcal{O}(n)}:F(\mathcal{O}(n)) \rightarrow G(\mathcal{O}(n))$ are defined such that the diagram
$$
\begin{CD}
F(\mathcal{O}(i)) & \overset{F(\psi)}{\rightarrow} & F(\mathcal{O}(i+1)) \\
@V{\underline{\Omega}_{\mathcal{O}(i)}}VV @VV{\underline{\Omega}_{\mathcal{O}(i+1)}}V \\
G(\mathcal{O}(i)) & \underset{G(\psi)}{\rightarrow} & G(\mathcal{O}(i+1))
\end{CD}
$$
commutes for all $i \in \mathbb{Z}$ and all $\psi \in {\sf Hom}_{\mathcal{O}_{\mathbb{P}^{1}}}(\mathcal{O}(i),\mathcal{O}(i+1))$.  Then there is a unique natural transformation $\Omega:F \rightarrow G$ such that $\Omega_{\mathcal{O}(n)}=\underline{\Omega}_{\mathcal{O}(n)}$ for all $n$.  Furthermore,
\begin{enumerate}
\item{} if $\underline{\Omega}_{\mathcal{O}(n)}$ is monic for all $n$ and $F$ vanishes on coherent torsion modules, then $\Omega$ is monic, and

\item{} if $\underline{\Omega}_{\mathcal{O}(n)}$ is epic for all $n$ then $\Omega$ is epic.
\end{enumerate}
\end{lemma}

\begin{proof}
By \cite[Lemma 7.5]{N}, to construct $\Omega$, it suffices to construct a natural transformation $\underline{\Omega}:F|_{{\sf coh} \mathbb{P}^{1}} \rightarrow G|_{{\sf coh} \mathbb{P}^{1}}$ where ${\sf coh }\mathbb{P}^{1}$ denotes the full subcategory of ${\sf Qcoh }\mathbb{P}^{1}$ consisting of coherent $\mathcal{O}_{\mathbb{P}^{1}}$-modules.  Uniqueness follows from \cite[Lemma 7.5]{N}.  We will see from the construction of $\underline{\Omega}$ below that if $\underline{\Omega}_{\mathcal{O}(n)}$ is monic for all $n$ and $F$ vanishes on coherent torsion modules, then $\underline{\Omega}$ is monic so that (1) follows from \cite[Lemma 7.5]{N}.  Similarly, we will see that if $\underline{\Omega}_{\mathcal{O}(n)}$ is epic for all $n$ then $\underline{\Omega}$ is epic so that (2) follows from \cite[Lemma 7.5]{N}.

If $\mathcal{T}$ is coherent torsion, we define $\underline{\Omega}_{\mathcal{T}}:F(\mathcal{T}) \rightarrow G(\mathcal{T})$ to be the zero map.

Next, we define $\underline{\Omega}_{\mathcal{F}}$ when
$\mathcal{F}$ is isomorphic to $\mathcal{O}(n)$.  Let $\alpha:
\mathcal{F} \rightarrow \mathcal{O}(n)$ be an isomorphism.  Define
$$
\underline{\Omega}_{\mathcal{F}}:=(G (\alpha))^{-1}
\underline{\Omega}_{\mathcal{O}(n)}  F (\alpha).
$$
If $\beta: \mathcal{F} \rightarrow \mathcal{O}(n)$ is another
isomorphism, then $\beta = b \alpha$ for some $0 \neq
b \in k$, whence $(G (\beta))^{-1}=b^{-1}(G
(\alpha))^{-1}$ and $F (\beta) = b F (\alpha)$; thus the
definition of $\Omega_{\mathcal{F}}$ does not depend on the choice
of $\alpha$.

Next, we define $\underline{\Omega}_{\mathcal{F}}$ for arbitrary $\mathcal{F}$ by writing $\mathcal{F}$ as a direct sum of indecomposables, say $\mathcal{F}=\oplus \mathcal{F}_{i}$, and defining $\underline{\Omega}_{\mathcal{F}}:=\oplus \underline{\Omega}_{\mathcal{F}_{i}}$.

To show that $\underline{\Omega}$ is a natural transformation we
must show that
\begin{equation} \label{eqn.welldefined}
\begin{CD}
F (\mathcal{F}) & \overset{F(\psi)}{\longrightarrow} & F (\mathcal{G}) \\
@V{\underline{\Omega}_{\mathcal{F}}}VV @VV{\underline{\Omega}_{\mathcal{G}}}V \\
G (\mathcal{F}) & \underset{G(\psi)}{\longrightarrow} & G (\mathcal{G})
\end{CD}
\end{equation}
commutes for all $\mathcal{F}$ and $\mathcal{G}$ and all maps
$\psi:\mathcal{F} \rightarrow \mathcal{G}$.  It suffices to check
this when $\mathcal{F}$ and $\mathcal{G}$ are indecomposable.  The
diagram commutes when $\mathcal{G}$ is torsion because $G(
\mathcal{G})=0$.  If $\mathcal{G}$ is torsion-free and
$\mathcal{F}$ torsion, then $\psi=0$ so the diagram commutes.  Thus,
the only remaining case is that when $\mathcal{F} \cong
\mathcal{O}(i)$ and $\mathcal{G} \cong \mathcal{O}(j)$ with $i
\leq j$.  The case $i=j$ is straightforward and we omit the verification in this case.
Thus, we may suppose $i<j$.

Write $\psi = \beta^{-1} \phi \alpha$ where $\alpha: \mathcal{F}
\rightarrow \mathcal{O}(i)$ and $\beta: \mathcal{G} \rightarrow
\mathcal{O}(j)$ are isomorphisms and $0 \neq \phi: \mathcal{O}(i)
\rightarrow \mathcal{O}(j)$.  Since $k$ is algebraically closed, we can write $\phi$ as a composition $\phi_{j}\phi_{j-1}
\cdots \phi_{i+1}$ where each $\phi_{l}:\mathcal{O}(l-1)
\rightarrow \mathcal{O}(l)$ is monic.  Now
$$
\underline{\Omega}_{\mathcal{O}(j)}  F(\phi_{j})
\cdots  F(\phi_{i+1}) = G (\phi_{j})  \cdots  G (
\phi_{i+1})  \underline{\Omega}_{\mathcal{O}(i)}
$$
and the fact that $\underline{\Omega}_{\mathcal{G}}  F(f)=G(f)  \underline{\Omega}_{\mathcal{F}}$ now follows from a straightforward computation, which we omit.
\end{proof}

We define a functor $R_{q}: {\sf Qcoh }\mathbb{P}^{1} \rightarrow {\sf Qcoh }(\operatorname{Spec }k)$ by
$$
R_{q}(-) := H^{0}(\mathbb{P}^{1},((-)/\mathcal{H}_{q}^{0}(-)) \otimes k(q)),
$$
where, $\mathcal{H}_{q}^{0}:{\sf Qcoh }\mathbb{P}^{1} \rightarrow {\sf Qcoh }\mathbb{P}^{1}$ is the functor which sends a quasi-coherent sheaf to its subsheaf with support at $q$.  In the case that the integer $m$ from Lemma \ref{lemma.alphabetaepi}(1) is greater than $0$, we will embed $R_{q}$ in $\Ker$ in the proof of Proposition \ref{prop.m}.

\begin{lemma} \label{lemma.k1}
The functor $R_{q}$
\begin{enumerate}
\item is $k$-linear,

\item commutes with direct limits, and

\item $R_{q}(\mathcal{T})=0$ for all coherent torsion modules $\mathcal{T}$.  Furthermore,

\item $R_{q}(\alpha_{n})$ is an isomorphism for all $n$, and

\item $R_{q}(\beta_{n})=0$ for all $n$.
\end{enumerate}
\end{lemma}

\begin{proof}
Since $\mathcal{H}_{q}^{0}$ is $k$-linear and since ${\sf Funct}_{k}({\sf Qcoh }X,{\sf Qcoh }Y)$ is abelian, $R_{q}$ is a composition of $k$-linear functors, and (1) follows.  To prove (2), we note that it suffices to prove $\mathcal{H}_{q}^{0}$ commutes with direct limits.  This is a straightforward exercise, which we omit.  Part (3) follows from the fact that the support of $\mathcal{T}/\mathcal{H}_{q}^{0}(\mathcal{T})$ does not include $q$.

We now prove (4).  Since $\mathcal{H}_{q}^{0}$ vanishes on vector bundles, $R_{q}(\alpha_{n})=H^{0}(\mathbb{P}^{1},\alpha_{n}\otimes k(q))$.  Therefore, it suffices to prove that $\alpha_{n}\otimes k(q)$ is an isomorphism.  To this end, since $p$ is disjoint from the support of $k(q)$, it follows that $-\otimes k(q)$ is exact when applied to
$0 \rightarrow \mathcal{O}(n-1) \overset{\alpha_{n}}{\rightarrow} \mathcal{O}(n) \rightarrow k(p) \rightarrow 0$
and the result follows.

To prove (5), we first note that, as in the proof of (4), $R_{q}(\beta_{n})=H^{0}(\mathbb{P}^{1},\beta_{n}\otimes k(q))$.  Therefore, it suffices to prove that $\beta_{n}\otimes k(q)=0$.  To this end, we apply $-\otimes k(q)$ to the short exact sequence $0 \rightarrow \mathcal{O}(n-1) \overset{\beta_{n}}{\rightarrow} \mathcal{O}(n) \rightarrow k(q) \rightarrow 0$ to obtain an exact sequence $\mathcal{O}(n-1)\otimes k(q) \overset{\beta_{n}\otimes k(q)}{\rightarrow} \mathcal{O}(n)\otimes k(q) \overset{\cong}{\rightarrow} k(q) \rightarrow 0$.  It follows that $\beta_{n}\otimes k(q)=0$ as desired.
\end{proof}

Recall that by Lemma \ref{lemma.alphabetaepi}(1), there exist integers $n_{0}$ and $m$ such that the dimension of $\Ker(\mathcal{O}(n))$ equals $m$ for all $n \geq n_{0}$.  This notation is employed in the following result, which is used in the proof of Proposition \ref{prop.m}.

\begin{lemma} \label{lemma.cokker}
Suppose $m \neq 0$.  Then there exists a closed point $r$ in the support of $W(F)$ corresponding to $\gamma_{0}:\mathcal{O}(-1) \rightarrow \mathcal{O}$ such that  $\operatorname{cok }(\Ker(\gamma_{n})) \neq 0$ for all $n>n_{0}$.
\end{lemma}

\begin{proof}
If not, then for an arbirary point $q$ in the support of $W(F)$, there exists an integer $n_{q}>n_{0}$ such that $\operatorname{cok }(\Ker(\beta_{n_{q}}))=0$ and so $\Ker(\beta_{n_{q}})$ is an isomorphism.  It follows from Lemma \ref{lemma.alphabetaepi}(2) that for any $q$ in the support of $W(F)$, $\Ker(\beta_{n_{0}+1})$ is an isomorphism.  On the other hand, by Lemma \ref{lemma.alphabetaepi}(1), $\Ker(\alpha_{n_{0}+1})$ is an isomorphism.  Thus, $\Ker(\delta)$ is an isomorphism for all nonzero $\delta \in \operatorname{Hom}_{\mathcal{O}_{\mathbb{P}^{1}}}(\mathcal{O}(n_{0}),\mathcal{O}(n_{0}+1))$.  However, if $x_{0}$, $x_{1}$ are indeterminates,
$$
\operatorname{det }(x_{0}\Ker(\alpha_{n_{0}+1})+x_{1}\Ker(\beta_{n_{0}+1}))
$$
is a homogeneous polynomial of degree $m>0$ in $x_{0},x_{1}$.  Therefore, it has a nontrivial zero.  Since $\Ker$ is $k$-linear, this provides a nonzero $\delta  \in \operatorname{Hom}_{\mathcal{O}_{\mathbb{P}^{1}}}(\mathcal{O}(n_{0}),\mathcal{O}(n_{0}+1))$ such that $\Ker(\delta)$ is not an isomorphism, and the contradiction establishes the result.
\end{proof}

\begin{prop} \label{prop.m}
The constant $m$ in Lemma \ref{lemma.alphabetaepi}(1) is $0$.  Therefore, there exist nonnegative integers $l_{i}$ and $l$ such that there is an exact sequence
\begin{equation} \label{eqn.finalses}
0 \rightarrow \oplus_{i=-l}^\infty {H}^{1}(\mathbb{P}^{1},(-)(i))^{\oplus l_{i}} \rightarrow F \overset{\Gamma_{F}}{\rightarrow} H^{0}(\mathbb{P}^{1},-\otimes W(F)) \rightarrow 0
\end{equation}
\end{prop}

\begin{proof}
The second part of the proposition follows from the first part in light of Corollary \ref{cor.cokvanish} and Corollary \ref{cor.kernelstructure}.

To prove the first part of the proposition, we proceed by contradiction.  Suppose $m \neq 0$.  We let $r$ denote the closed point in the support of $W(F)$ shown to exist in Lemma \ref{lemma.cokker}.  We show that there is a monic natural transformation $\Delta:R_{r} \rightarrow \Ker$.  This contradicts the fact that if $u:U \rightarrow \mathbb{P}^{1}$ is inclusion of an affine open subscheme containing $r$ then $R_{r}(u_{*}u^{*}\mathcal{O}_{\mathbb{P}^{1}}) \neq 0$ while, since $\Ker$ is totally global, $\Ker (u_{*}u^{*}\mathcal{O}_{\mathbb{P}^{1}})=0$.

The functor $R_{r}$ is $k$-linear and commutes with direct limits by Lemma \ref{lemma.k1}(1) and (2), and $\Ker$ is $k$-linear and commutes with direct limits since it is a kernel of a natural transformation between $k$-linear functors which commute with direct limits by Corollary \ref{cor.cokvanish}.  Furthermore, if $\mathcal{T}$ is coherent torsion, then $R_{r}(\mathcal{T})=0$ by Lemma \ref{lemma.k1}(3).  Therefore, to define a monomorphism $\Delta:R_{r} \rightarrow \Ker$, it suffices, by Lemma \ref{lemma.inductive}, to define, for all $n \in \mathbb{Z}$, a monomorphism $\underline{\Delta}_{\mathcal{O}(n)}:R_{r}(\mathcal{O}(n)) \rightarrow \Ker(\mathcal{O}(n))$ such that
\begin{equation} \label{eqn.mainsquare}
\begin{CD}
R_{r}(\mathcal{O}(i)) & \overset{R_{r}(\psi)}{\rightarrow} & R_{r}(\mathcal{O}(i+1)) \\
@V{\underline{\Delta}_{\mathcal{O}(i)}}VV @VV{\underline{\Delta}_{\mathcal{O}(i+1)}}V \\
\Ker(\mathcal{O}(i)) & \underset{\Ker(\psi)}{\rightarrow} & \Ker(\mathcal{O}(i+1))
\end{CD}
\end{equation}
commutes for all $i \in \mathbb{Z}$ and all $\psi \in \operatorname{Hom}_{\mathcal{O}_{\mathbb{P}^{1}}}(\mathcal{O}(i),\mathcal{O}(i+1))$. In our construction of $\underline{\Delta}$, we will retain the notation both preceding the statement of Lemma \ref{lemma.cokker} and defined in the statement of Lemma \ref{lemma.cokker}.

We begin by defining $\underline{\Delta}_{\mathcal{O}(n_{0})}$.  To this end we let $a_{n_{0}}$ be a generator of $R_{r}(\mathcal{O}(n_{0}))$ and we let $0 \neq b_{n_{0}} \in \Ker(\mathcal{O}(n_{0}))$ be in the kernel of $\Ker(\gamma_{n_{0}+1})$.  Such an element exists by Lemma \ref{lemma.cokker}.  We define $\underline{\Delta}_{\mathcal{O}(n_{0})}(a_{n_{0}})=b_{n_{0}}$ and extend linearly.

We define $\underline{\Delta}_{\mathcal{O}(n_{0}+1)}$ as follows:  we let $a_{n_{0}+1} \in R_{r}(\mathcal{O}(n_{0}+1))$ be $R_{r}(\alpha_{n_{0}+1})(a_{n_{0}})$.  Since $R_{r}(\alpha_{n_{0}+1})$ is an isomorphism by Lemma \ref{lemma.k1}(4), $a_{n_{0}+1} \neq 0$ so is a basis for $R_{r}(\mathcal{O}(n_{0}+1))$.  We define $\underline{\Delta}_{\mathcal{O}(n_{0}+1)}(a_{n_{0}+1})=\Ker(\alpha_{n_{0}+1})(b_{n_{0}})=:b_{n_{0}+1}$ and extend linearly.  The fact that $\underline{\Delta}_{\mathcal{O}(n_{0}+1)}$ is monic follows immediately from the fact that, by choice of $n_{0}$, $\Ker(\alpha_{n_{0}+1})$ is an isomorphism by Lemma \ref{lemma.alphabetaepi}(1).

We now check that (\ref{eqn.mainsquare}) commutes in case $i=n_{0}$.  Since, by construction, the diagram commutes when $\psi=\alpha_{n_{0}+1}$, we need only check that it commutes when $\psi=\gamma_{n_{0}+1}$.  By Lemma \ref{lemma.k1}(5), $R_{r}(\gamma_{n_{0}+1})=0$ and $b_{n_{0}}$ is chosen so that $\Ker(\gamma_{n_{0}+1})(b_{n_{0}})=0$, so the diagram commutes for all $\psi \in \operatorname{Hom}_{\mathcal{O}_{\mathbb{P}^{1}}}(\mathcal{O}(n_{0}),\mathcal{O}(n_{0}+1))$ as desired.

Now suppose $n>n_{0}$ and for all $j$ such that $n_{0} < j \leq n$, we have defined a monic $\underline{\Delta}_{\mathcal{O}(j)}$ such that (\ref{eqn.mainsquare}) commutes when $i=j-1$.

We define $\underline{\Delta}_{\mathcal{O}(n+1)}$ as follows.  Let $a_{n+1} \in R_{r}(\mathcal{O}(n+1))$ be $R_{r}(\alpha_{n+1})(a_{n})$ where $a_{n}$ is some nonzero element of $R_{r}(\mathcal{O}(n))$.  Since $R_{r}(\alpha_{n+1})$ is an isomorphism by Lemma \ref{lemma.k1}(4), $a_{n+1} \neq 0$ so is a basis for $R_{r}(\mathcal{O}(n+1))$.  If $b_{n}=\underline{\Delta}_{\mathcal{O}(n)}(a_{n})$, then by the induction hypothesis, $b_{n} \neq 0$ and $b_{n} = \Ker(\alpha_{n})(b_{n-1})$ for some $b_{n-1} \in \operatorname{ker }(\Ker(\gamma_{n}))$.  We define $\underline{\Delta}_{\mathcal{O}(n+1)}(a_{n+1})=\Ker(\alpha_{n+1})(b_{n})$ and extend linearly.  Since $\Ker(\alpha_{n+1})$ is an isomorphism by Lemma \ref{lemma.alphabetaepi}(1), $\underline{\Delta}_{\mathcal{O}(n+1)}$ is monic.

We check that (\ref{eqn.mainsquare}) commutes when $i=n$ and $\psi=\gamma_{n+1}$, from which it will follow that the diagram commutes for all $\psi \in \operatorname{Hom}_{\mathcal{O}_{\mathbb{P}^{1}}}(\mathcal{O}(n),\mathcal{O}(n+1))$.  By Lemma \ref{lemma.k1}(5), $R_{r}(\gamma_{n+1})=0$.  Thus, we must show that $\Ker(\gamma_{n+1})(b_{n})=0$.  To show this, we note that the exactness of (\ref{eqn.kerapplied}) implies the sequence
\begin{equation} \label{eqn.complex}
\Ker (\mathcal{O}(n-1)) \overset{(\Ker(\alpha_{n}),-\Ker(\gamma_{n}))}{\longrightarrow} \Ker (\mathcal{O}(n))^{\oplus 2}
\end{equation}
$$
\overset{\Ker(\gamma_{n+1})+\Ker(\alpha_{n+1})}{\longrightarrow} \Ker (\mathcal{O}(n+1)) \rightarrow 0
$$
is exact.  Thus, since $(b_{n},0)=(\Ker(\alpha_{n}),-\Ker(\gamma_{n}))(b_{n-1}) \in \Ker (\mathcal{O}(n))^{\oplus 2}$, we must have
$$
\Ker(\gamma_{n+1})(b_{n})+\Ker(\alpha_{n+1})(0)=0 \in \Ker (\mathcal{O}(n+1))
$$
so that $\Ker(\gamma_{n+1})(b_{n})=0$ as desired.

We now define $\underline{\Delta}_{\mathcal{O}(n)}$ for all $n<n_{0}$.  We begin by defining $\underline{\Delta}_{\mathcal{O}(n_{0}-1)}$.  To this end, we let $a_{n_{0}-1}=R_{r}(\alpha_{n_{0}})^{-1}(a_{n_{0}})$, which makes sense by Lemma \ref{lemma.k1}(4).  We claim there exists an element
$$
b_{n_{0}-1} \in \Ker(\alpha_{n_{0}})^{-1}(b_{n_{0}}) \cap \operatorname{ker }(\Ker (\gamma_{n_{0}}))
$$
and we define $\underline{\Delta}_{\mathcal{O}(n_{0}-1)}(a_{n_{0}-1})=b_{n_{0}-1}$ and extend linearly.  We note that the claim will imply that $\underline{\Delta}_{\mathcal{O}(n_{0}-1)}$ is monic and that (\ref{eqn.mainsquare}) commutes in case $i=n_{0}-1$.  To prove the claim, we note that (\ref{eqn.complex}) is exact for any value of $n$.  In particular, when $n=n_{0}$, exactness of (\ref{eqn.complex}) implies that $(b_{n_{0}},0) \in  \Ker (\mathcal{O}(n_{0}))^{\oplus 2}$ must be of the form $(\Ker(\alpha_{n_{0}}),-\Ker(\gamma_{n_{0}}))(b_{n_{0}-1})$ for some $b_{n_{0}-1} \in \Ker (\mathcal{O}(n_{0}-1))$ whence the claim.

Finally, suppose $n<n_{0}$ and for all $j$ such that $n \leq j < n_{0}$, we have defined monic $\underline{\Delta}_{\mathcal{O}(j)}$ such that (\ref{eqn.mainsquare}) commutes in case $i=j$.

The definition of $\underline{\Delta}_{\mathcal{O}(n-1)}$ as well as the proof that it is monic and makes (\ref{eqn.mainsquare}) commute in case $i=n-1$ is identical to the proof of these properties for $\underline{\Delta}_{\mathcal{O}(n_{0}-1)}$, and we omit it.  The proposition follows.
\end{proof}

\section{Proof of the structure theorem}
In this section, we assume $F \in {\sf bimod}_{k}(\mathbb{P}^{1}-\operatorname{Spec }k)$.  By Proposition \ref{prop.coherent}, we know $W(F)$ is coherent torsion.
\begin{thm} \label{thm.extension}
There exist nonnegative integers $l$, $l_{i}$ such that
$$
F \cong \oplus_{i=-l}^\infty {H}^{1}(\mathbb{P}^{1},(-)(i))^{\oplus l_{i}} \oplus H^{0}(\mathbb{P}^{1},-\otimes W(F)).
$$
\end{thm}

\begin{proof}
In case $W(F)=0$, Theorem \ref{thm.watts} implies that $F \cong \Ker$.  Therefore, the result follows from \cite[Lemma 7.2]{N} and \cite[Theorem 7.12]{N}.

If $W(F) \neq 0$, Proposition \ref{prop.m} applies.  We prove that the short exact sequence (\ref{eqn.finalses}) splits, i.e. we construct a natural transformation $\Lambda:F \rightarrow \Ker$ which splits the inclusion $\Theta: \Ker \rightarrow F$ from Theorem \ref{thm.watts} in the category ${\sf Funct}_{k}({\sf Qcoh }\mathbb{P}^{1},{\sf Qcoh }(\operatorname{Spec }k))$.  The functors $F$ and $\Ker$ are $k$-linear and commute with direct limits and $\Ker$ is totally global.  Thus, to construct $\Lambda$ it suffices, by Lemma \ref{lemma.inductive}, to define, for all $n \in \mathbb{Z}$, a morphism $\underline{\Lambda}_{\mathcal{O}(n)}$ such that $\underline{\Lambda}_{\mathcal{O}(i)} \Theta_{\mathcal{O}(i)}$ is the identity for all $i \in \mathbb{Z}$ and such that the diagram
\begin{equation} \label{eqn.taucommute}
\begin{CD}
F(\mathcal{O}(i)) & \overset{F\psi}{\rightarrow} & F(\mathcal{O}(i+1)) \\
@V{\underline{\Lambda}_{\mathcal{O}(i)}}VV @VV{\underline{\Lambda}_{\mathcal{O}(i+1)}}V \\
\Ker(\mathcal{O}(i)) & \underset{\Ker (\psi)}{\rightarrow} & \Ker(\mathcal{O}(i+1))
\end{CD}
\end{equation}
commutes for all $i \in \mathbb{Z}$ and all $\psi \in \operatorname{Hom}_{\mathcal{O}_{\mathbb{P}^{1}}}(\mathcal{O}(i),\mathcal{O}(i+1))$.  Once $\Lambda$ is constructed, we will show that $\Lambda \Theta$ is the identity.

By Proposition \ref{prop.m}, there exists $n_{0} \in \mathbb{Z}$ such that $\Ker (\mathcal{O}(n)) = 0$ for all $n \geq n_{0}$.  For such $n$, we define $\underline{\Lambda}_{\mathcal{O}(n)}=0$.

We now proceed to define $\underline{\Lambda}_{\mathcal{O}(n)}$ for $n<n_{0}$ inductively.  We begin by defining $\underline{\Lambda}_{\mathcal{O}(n_{0}-1)}$.  To this end, we pick a subspace $B_{n_{0}-1} \subset F(\mathcal{O}(n_{0}-1))$ complimentary to the image of $\Theta_{\mathcal{O}(n_{0}-1)}$.  If $a \in \Ker(\mathcal{O}(n_{0}-1))$ and $b \in B_{n_{0}-1}$ then we define $\underline{\Lambda}_{\mathcal{O}(n_{0}-1)}(\Theta_{\mathcal{O}(n_{0}-1)}(a)+b)=a$.  It follows immediately that $\underline{\Lambda}_{\mathcal{O}(n_{0}-1)} \Theta_{\mathcal{O}(n_{0}-1)}$ is the identity and that (\ref{eqn.taucommute}) commutes in case $i=n_{0}-1$.

Next, suppose there is an $n<n_{0}$ such that for all $j$ with $n \leq j \leq n_{0}$ we have defined $\underline{\Lambda}_{\mathcal{O}(j)}$ with the property that $\underline{\Lambda}_{\mathcal{O}(j)} \Theta_{\mathcal{O}(j)}$ is the identity and makes (\ref{eqn.taucommute}) commute when $i=j$.

We now define $\underline{\Lambda}_{\mathcal{O}(n-1)}$.  We do this by constructing a subspace $B_{n-1}$ of $F(\mathcal{O}(n-1))$ complementary to $\Theta_{\mathcal{O}(n-1)}(\Ker (\mathcal{O}(n-1)))$ and defining $\underline{\Lambda}_{\mathcal{O}(n-1)}(\Theta_{\mathcal{O}(n-1)}(a)+b)=a$, where $b \in B_{n-1}$.  We will then see that, by choice of $B_{n-1}$, $\underline{\Lambda}_{\mathcal{O}(n-1)}$ makes (\ref{eqn.taucommute}) commute when $i=n-1$.

Let $B_{n} \subset F(\mathcal{O}(n))$ denote the kernel of $\underline{\Lambda}_{\mathcal{O}(n)}$.  We define $B_{n-1} \subset F(\mathcal{O}(n-1))$ as follows.  We note that the image under $\Theta_{\mathcal{O}(n-1)}$ of the kernel of
$$
(\Ker (\alpha_{n}),-\Ker (\beta_{n})):\Ker (\mathcal{O}(n-1)) \rightarrow \Ker (\mathcal{O}(n))^{\oplus 2},
$$
which we denote by $K$, is contained in $F\alpha_{n}^{-1}(B_{n}) \cap F\beta_{n}^{-1}(B_{n})$.  We define $B_{n-1} \subset F\alpha_{n}^{-1}(B_{n}) \cap F\beta_{n}^{-1}(B_{n})$ to be a complimentary subspace to $K$ in $F\alpha_{n}^{-1}(B_{n}) \cap F\beta_{n}^{-1}(B_{n})$.

We claim that $B_{n-1}$ is complimentary to the image of $\Theta_{\mathcal{O}(n-1)}$.  To prove the claim, consider the exact sequence
\begin{equation} \label{eqn.vectorlast}
0 \longrightarrow \mathcal{O}(n-1) \overset{(\alpha_{n},-\beta_{n})}{\longrightarrow} \mathcal{O}(n)^{\oplus 2} \overset{\beta_{n+1}+\alpha_{n+1}}{\longrightarrow} \mathcal{O}(n+1) \longrightarrow 0.
\end{equation}
The sequence (\ref{eqn.vectorlast}) induces a diagram
\begin{equation} \label{eqn.tworow}
\begin{CD}
\Ker (\mathcal{O}(n-1)) & \rightarrow & \Ker (\mathcal{O}(n))^{\oplus 2} & \rightarrow & \Ker(\mathcal{O}(n+1)) & \rightarrow & 0  \\
@VVV @VVV @VVV \\
F(\mathcal{O}(n-1)) & \rightarrow & F(\mathcal{O}(n))^{\oplus 2} & \rightarrow & F(\mathcal{O}(n+1)) & \rightarrow & 0
\end{CD}
\end{equation}
whose verticals are inclusions and whose top row is exact by the exactness of (\ref{eqn.snake}) in case $\Upsilon=\Gamma_{F}$ and (\ref{eqn.ses}) is (\ref{eqn.vectorlast}).  We begin the proof of the claim by showing that $B_{n-1}$ intersected with image of $\Theta_{\mathcal{O}(n-1)}$ is $0$.  Suppose $x$ is in the intersection.  By the commutativity of (\ref{eqn.tworow}), $F\alpha_{n}(x) \in B_{n} \cap \Theta_{\mathcal{O}(n)}(\Ker(\mathcal{O}(n)))$.  Thus, $F\alpha_{n}(x)=0$ and similarly, $F\beta_{n}(x)=0$.  Since the middle vertical of (\ref{eqn.tworow}) is an inclusion, and since $x \in \Theta_{\mathcal{O}(n-1)}(\Ker(\mathcal{O}(n-1)))$, say $x=\Theta_{\mathcal{O}(n-1)}(a)$, it follows again from the commutativity of (\ref{eqn.tworow}) that $\Ker (\alpha_{n})(a)=0=\Ker (\beta_{n})(a)$.  Thus, $x$ is in the image, under $\Theta_{\mathcal{O}(n-1)}$, of the kernel of the top left horizontal.  Therefore, by choice of $B_{n-1}$, $x=0$.

To complete the proof of the claim, it remains to show that
$$
\Theta_{\mathcal{O}(n-1)}(\Ker (\mathcal{O}(n-1))) + B_{n-1}= F\mathcal{O}(n-1).
$$
To this end, it suffices to show
\begin{equation} \label{eqn.end}
\Theta_{\mathcal{O}(n-1)}(\Ker (\mathcal{O}(n-1)))+F\alpha_{n}^{-1}(B_{n}) \cap F\beta_{n}^{-1}(B_{n})=F(\mathcal{O}(n-1)).
\end{equation}
To prove (\ref{eqn.end}), suppose $x \in F(\mathcal{O}(n-1))$.  Then $(F\alpha_{n},-F\beta_{n})(x)$ has the form
$$
(\Theta_{\mathcal{O}(n)}(a_1)+b_{1},\Theta_{\mathcal{O}(n)}(a_2)+b_{2})\in F(\mathcal{O}(n))^{\oplus 2}
$$
where $a_1,a_2 \in \Ker (\mathcal{O}(n))$ and $b_1,b_2 \in B_{n}$.  Since the bottom row of (\ref{eqn.tworow}) is exact, $(\Theta_{\mathcal{O}(n)}(a_1)+b_1,\Theta_{\mathcal{O}(n)}(a_2)+b_2)$ is in the kernel of the second map in the bottom row.  By the induction hypothesis, the diagram (\ref{eqn.taucommute}) commutes in case $i=n$.  It follows that $(a_1,a_2) \in \Ker (\mathcal{O}(n))^{\oplus 2}$ is in the kernel of the second top horizontal of (\ref{eqn.tworow}).  Therefore, there exists a $c \in \Ker(\mathcal{O}(n-1))$ which maps to $(a_1,a_2) \in \Ker (\mathcal{O}(n))^{\oplus 2}$ via the top left horizontal of (\ref{eqn.tworow}) so that $(F\alpha_{n},-F\beta_{n})(x-\Theta_{\mathcal{O}(n-1)}(c))=(b_1,b_2) \in F(\mathcal{O}(n))^{\oplus 2}$.  We conclude that $x-\Theta_{\mathcal{O}(n-1)}(c) \in F\alpha_{n}^{-1}(B_{n}) \cap F\beta_{n}^{-1}(B_{n})$, establishing the claim that $B_{n-1}$ is complementary to the image of $\Theta_{\mathcal{O}(n-1)}$ in $F(\mathcal{O}(n-1))$.

We define $\underline{\Lambda}_{\mathcal{O}(n-1)}(\Theta_{\mathcal{O}(n-1)}(a)+b)=a$, where $b \in B_{n-1}$.  It follows immediately that $\underline{\Lambda}_{\mathcal{O}(n-1)} \Theta_{\mathcal{O}(n-1)}$ is the identity and that (\ref{eqn.taucommute}) commutes in the case that $i=n-1$.

 To show that $\Lambda \Theta$ is the identity, we note that by construction of $\Lambda$, $(\Lambda \Theta)_{\mathcal{M}}=\operatorname{id }_{\Ker(\mathcal{M})}$ for all coherent $\mathcal{M}$.  Thus, by the uniqueness statement of Lemma \ref{lemma.inductive}, the composition $\Lambda \Theta$ is the identity functor.  The result follows.
\end{proof}

\begin{cor} \label{cor.vectorbundles}
The following are equivalent:

\begin{enumerate}

\item The natural transformation $\Gamma_{F}$ is an isomorphism.

\item $\operatorname{dim}_{k}(F(\mathcal{O}(i)))$ is constant.

\item $F$ is exact on short exact sequences of vector bundles.

\end{enumerate}
\end{cor}

\begin{proof}
We first prove (1) if and only if (2):  if $\Gamma_{F}$ is an isomorphism, then $F \cong H^{0}(\mathbb{P}^{1},-\otimes W(F))$.  Therefore, since $W(F)$ is coherent torsion, $\operatorname{dim}_{k}(F(\mathcal{O}(i)))$ is independent of $i$. Conversely, if $\operatorname{dim}_{k}(F(\mathcal{O}(i)))$ is constant, Theorem \ref{thm.extension} implies $\Ker =0$.

Next we prove (1) if and only if (3):  If $\Gamma_{F}$ is an isomorphism then, since $W(F)$ is coherent torsion by Proposition \ref{prop.coherent}, $F$ is exact on vector bundles.  Conversely, since Theorem \ref{thm.extension} implies there exist nonnegative integers $n$, $n_{i}$ such that $F \cong \oplus_{i=-n}^\infty {H}^{1}(\mathbb{P}^{1},(-)(i))^{\oplus n_{i}} \oplus H^{0}(\mathbb{P}^{1},-\otimes W(F))$, $F$ exact on vector bundles implies that $n_{i}=0$ for all $i$.  It follows that $\Gamma_{F}$ is an isomorphism.
\end{proof}

\begin{cor} \label{cor.morphism}
The following are equivalent:
\begin{enumerate}
\item $F \cong f^{*}$ where $f:\operatorname{Spec }k \rightarrow \mathbb{P}^{1}$ is a morphism of $k$-schemes.

\item  $\operatorname{dim}_{k}(F(\mathcal{O}(i)))=1$ for all $i$.

\item $F$ is exact on short exact sequences of vector bundles and $\operatorname{dim}_{k}(F(\mathcal{O}(i)))=1$ for some $i$.
\end{enumerate}
\end{cor}

\begin{proof}
We first prove (1) if and only if (2):  if $F \cong f^{*}$ where $f:\operatorname{Spec }k \rightarrow \mathbb{P}^{1}$ is a morphism of schemes over $k$, then $f^{*} = H^{0}(\mathbb{P}^{1},-\otimes k(r))$ for some closed point $r$ in $\mathbb{P}^{1}$, whence the forward direction.  Conversely, by Theorem \ref{thm.extension}, there exist nonnegative integers $n$, $n_{i}$ such that $F \cong \oplus_{i=-n}^\infty {H}^{1}(\mathbb{P}^{1},(-)(i))^{\oplus n_{i}} \oplus H^{0}(\mathbb{P}^{1},-\otimes W(F))$.  Therefore, if $\operatorname{dim}_{k}(F(\mathcal{O}(i)))=1$ for all $i$, $n_{i}=0$ for all $i$, and $\operatorname{dim}_{k}(\mathcal{O}(i) \otimes W(F))=1$ for all $i$.  Since $W(F)$ is coherent torsion, $W(F) \cong k(r)$ for some closed point $r \in \mathbb{P}^{1}$.

Next we prove (2) if and only if (3):  if $\operatorname{dim}_{k}(F(\mathcal{O}(i)))=1$ for all $i$, then Corollary \ref{cor.vectorbundles} implies that $F$ is exact on vector bundles and trivially implies the second part of (3).  Conversely, if $F$ is exact on vector bundles, Corollary \ref{cor.vectorbundles} implies that $\operatorname{dim}_{k}(F(\mathcal{O}(i)))$ is constant.  Since it equals 1 for some $i$, it equals 1 for all $i$, whence (2).
\end{proof}

\end{document}